\documentclass{amsart}
\usepackage{amsmath,amsfonts,mathrsfs,enumerate,graphicx,subfig,wasysym}
\usepackage{color} 
\newtheorem{theorem}{Theorem}

\newtheorem{proposition}{Proposition}[section]
\newtheorem{corollary}[proposition]{Corollary}
\newtheorem{lemma}[proposition]{Lemma}

\newcommand{\triple}[1]{{\left\vert\kern-0.25ex\left\vert\kern-0.25ex\left\vert #1 
    \right\vert\kern-0.25ex\right\vert\kern-0.25ex\right\vert}}

\newcommand{\I}{\mathcal{I}}

\newcommand{\threebar}[1]{{\left\vert\kern-0.25ex\left\vert\kern-0.25ex\left\vert #1 
    \right\vert\kern-0.25ex\right\vert\kern-0.25ex\right\vert}}

\newcommand{\R}{\mathbb{R}}
\newcommand{\A}{\mathsf{A}}

\newcommand{\C}{\mathbb{C}}
\newcommand{\B}{\mathsf{B}}
\newcommand{\N}{\mathbb{N}}

\newcommand{\Z}{\mathbb{Z}}

\title[Stability of two-dimensional discrete linear switching systems]{A stability dichotomy for discrete-time linear switching systems in dimension two}
\author{Ian D. Morris}

\address{I. D. Morris: School of Mathematical Sciences, Queen Mary University of London, Mile End Road, London E1 4NS, United Kingdom}
\email{i.morris@qmul.ac.uk }
\begin{document}

\begin{abstract}
We prove that for every discrete-time linear switching system in two complex variables and with finitely many switching states, either the system is Lyapunov stable or there exists a trajectory which escapes to infinity with at least linear speed. We also give a checkable algebraic criterion to distinguish these two cases. This dichotomy was previously known to hold for systems in two real variables, but is known to be false in higher dimensions and for systems with infinitely many switching states.
\end{abstract}
\maketitle

\section{Introduction}
If $\A=\{A_i \colon i \in \I\}$ is a bounded set of $d\times d$ real or complex matrices, a trajectory of the \emph{discrete-time linear switching system} defined by $\A$ is a sequence of vectors $(v_n)_{n=0}^\infty$ such that $v_{n+1}=A_{\sigma(n)}v_n$ for all $n \geq 0$. Here the sequence $\sigma \colon \N \to \I$ is referred to as the \emph{switching law} and $v_0$ the \emph{initial vector}.  The elements of $\A$ are called the \emph{switching states} of $\A$ and in most cases we will assume them to be finite in number. Clearly we may write $v_n=A_{\sigma(n)}\cdots A_{\sigma(1)}v_0$ for every $n \geq 1$ and we will prefer this formulation in the sequel. These systems arise as the discrete-time analogues of the more widely studied continuous-time linear switching system, which is the differential equation $v'(t)=A(t)v(t)$ a.e. where $A\colon [0,\infty) \to \A$ is measurable and $v \colon [0,\infty) \to \R^d$ is Lipschitz continuous. 

This note is concerned with the worst-case stability properties of discrete-time linear switching systems under arbitrary switching laws. Given a bounded nonempty set of $d \times d$ matrices $\A=\{A_i \colon i \in \I\}$ one may distinguish four distinct stability regimes (as noted in, for example, \cite{ChMaSi12,Su08}):
\begin{enumerate}[1)]
\item
\emph{Exponential stability:} there exist $C>0$ and $\kappa>0$ such that for every initial vector $v_0$ and switching law $\sigma$,
\[\left\|A_{\sigma(n)} \cdots A_{\sigma(1)}v_0\right\| \leq Ce^{-\kappa n}\left\|v_0\right\|\]
for every $n \geq 0$.
\item
\emph{Exponential instability:} there exist $C>0$, $\kappa>0$, an initial vector $v_0$ and a switching law $\sigma$ such that 
\[\left\|A_{\sigma(n)} \cdots A_{\sigma(1)}v_0\right\| \geq Ce^{\kappa n}\left\|v_0\right\|\]
for every $n \geq 0$.
\item
\emph{Marginal stability:}  there exists $C>0$ such that for every initial vector $v_0$ and switching law $\sigma$,
\[\left\|A_{\sigma(n)} \cdots A_{\sigma(1)}v_0\right\| \leq C \left\|v_0\right\|,\]
but exponential stability does not hold. In this case there necessarily exists a trajectory which does not converge to the origin.
\item
\emph{Marginal instability:} the system is not exponentially unstable, but there exists an unbounded trajectory.
\end{enumerate}
The four categories are mutually exclusive and every $\A$ belongs to exactly one of the four categories. We remark that the system is Lyapunov stable if and only if it is either exponentially stable or marginally stable. It is by now well known (see for example \cite{Ju09}) that exponential stability and exponential instability are characterised by the \emph{joint spectral radius} of $\A$,
\[\varrho(\A):=\lim_{n \to \infty}\sup_{i_1,\ldots,i_n \in \I} \left\|A_{i_n}\cdots A_{i_1}\right\|^{\frac{1}{n}} = \inf_{n \geq 1}\sup_{i_1,\ldots,i_n \in \I} \left\|A_{i_n}\cdots A_{i_1}\right\|^{\frac{1}{n}}.\]
Specifically, if $\varrho(\A)<1$ then exponential stability holds, and if $\varrho(\A)>1$ then exponential instability holds. If $\varrho(\A)=1$ the system is therefore either marginally stable or marginally unstable, but distinguishing the two cases in practice has proven quite difficult (see for example \cite{BlTs00,Mo22}). The separate problem of computing the joint spectral radius itself forms the subject of a substantial body of research: see for example \cite{GuPr13,HaMoSiTh11,Ju09,VaHeJu14,WaMaMa21} and references therein. 

The above four definitions leave open the question of what the rate of growth of trajectories of a  marginally unstable system can actually be. Given a bounded nonempty set of $d\times d$ matrices $\A=\{A_i \colon i \in \I\}$ such that $\varrho(\A)=1$, one may define the \emph{rate of marginal instability} to be the sequence
\[n \mapsto \sup_{\sigma \colon \N \to \I} \sup_{\|v_0\|=1} \left\|A_{\sigma(n)} \cdots A_{\sigma(1)}v_0\right\|\]
or more simply
\[n \mapsto \sup_{i_1,\ldots,i_n\in \I} \left\|A_{i_n} \cdots A_{i_1}\right\|\]
which is the best possible uniform upper bound on the growth of trajectories whose initial vector lies in the unit ball. If $\A$ has only one switching state $A$ then the above sequence is simply the sequence $n \mapsto \|A^n\|$, and it is a simple consequence of the Jordan form theorem that (assuming $\varrho(\A)=1$) we have $\|A^n\| \sim n^k$ for some non-negative integer $k<d$. Outside of this special case the behaviour of the rate of marginal instability is rather less clear, and a by-now significant body of work has attempted to characterise its behaviour (see for example \cite{Be05,BeCoHa16,ChMaSi12,GuZe01,JuPrBl08,Mo22,PrJu15,VaMo22}). In general the problems both of distinguishing marginal stability from marginal instability, and understanding the range of possible behaviours of the rate of marginal instability, remain wide open. 

The purpose of this note is to demonstrate that in the specific context of two-dimensional linear switching systems over $\C$ with finitely many switching states, both of these two questions have straightforward answers. In the case of such a system which is known to be either marginally stable or marginally unstable, we will show that the two possibilities can be distinguished from one another by a simple algebraic criterion. Moreover, in the case where the system is marginally unstable, we will show that its rate of marginal instability is asymptotically bounded above and below by a linear sequence.
\begin{theorem}\label{th:main}
Let $\A=\{A_1, A_2,\ldots,A_N\}$ be a finite set of $2 \times 2$ complex matrices such that $\varrho(\A)=1$. Then: 
\begin{enumerate}[(i)]
\item
If the matrices $A_1,\ldots,A_N$ do not share a common eigenvector, or if the set $\{ A_j \colon |\det A_j|=1\}$ is simultaneously diagonalisable, then there exists a constant $C>0$ such that for every switching sequence $\sigma \colon \N \to \{1,\ldots,N\}$ and vector $v \in \C^2$ we have $ \|A_{\sigma(n)}\cdots A_{\sigma(1)}v\| \leq C\|v\|$ for all $n \geq 1$.
\item
If the matrices $A_1,\ldots,A_N$ share a common eigenvector and additionally the set $\{ A_j \colon |\det A_j|=1\}$ is \emph{not} simultaneously diagonalisable then there exist a switching sequence $\hat\sigma \colon \N \to \{1,\ldots,N\}$ and a nonzero vector $v \in \C^2$
such that
\begin{equation}\label{eq:liminf}\liminf_{n \to \infty} \frac{1}{n} \|A_{\hat\sigma(n)}\cdots A_{\hat\sigma(1)}v\| >0.\end{equation}
The sequence $\hat\sigma$ may be chosen so that $\{A_{\hat\sigma(n)} \colon n \geq 1\}$ is not simultaneously diagonalisable, such that $|\det A_{\hat\sigma(n)}|=1$ for every $n \geq 1$, and such that $\hat\sigma$ either is constant or takes exactly two distinct values. Furthermore for all switching sequences $\sigma \colon \N \to \{1,\ldots,N\}$ and vectors $v \in \C^2$,
\[\|A_{\sigma(n)}\cdots A_{\sigma(1)}v\| \leq \left(1+n\cdot \max_j \|A_j\|\right)\|v\|\]
for all $n \geq 1$.
\end{enumerate}
\end{theorem}
The proof of Theorem \ref{th:main} which we present is in principle constructive. In the cases where the system is marginally stable we are able to give explicit uniform upper bounds on the growth of trajectories, and in the marginally unstable case we construct the escaping trajectory by an explicit inductive procedure. The limit inferior \eqref{eq:liminf} may also be made explicit in principle. However, the precise values of these upper and lower bounds depend on constants induced by a change of basis for $\C^2$ and also depend on whether or not certain ratios of eigenvalues are roots of unity. These explicit results are thus relatively complicated to incorporate into the statement of Theorem \ref{th:main} and we therefore leave them in the form of individual propositions, listed in the next section, each of which corresponds to the proof of one of the major sub-cases of Theorem \ref{th:main}. We do not attempt to estimate the value of the limit inferior \eqref{eq:liminf} in this note.

The following consequence of Theorem \ref{th:main} is straightforward:
\begin{corollary}
Let $\B=\{B_1, B_2,\ldots,B_N\}$ be a finite set of $2 \times 2$ complex matrices and suppose that $\varrho(\B)\neq 0$. Define
\[b_n:=\max_{1 \leq i_1,\ldots,i_n \leq N} \|B_{i_1}\cdots B_{i_n}\|\]
for every $n\geq 1$. Then either the sequence $b_n/\varrho(\B)^n$ is bounded away from zero and infinity, or the sequence $b_n/(n\varrho(\B)^n)$ is bounded away from zero and infinity.
\end{corollary}

Theorem \ref{th:main} is ``fragile'' in the sense that its conclusions can become false under only minor modifications to the hypotheses. If the number of switching states is allowed to be infinite then it is known that the rate of marginal instability can be asymptotic to a polynomial sequence $n^\alpha$ with $\frac{1}{2}<\alpha<1$ (see \cite{GuZe01}) 
and therefore no dichotomy between boundedness and linear growth occurs. If the dimension of the matrices is raised to three or higher but the number of switching states is kept finite then similar phenomena may occur (see \cite{PrJu15,VaMo22}). In three dimensions it is moreover shown in \cite{Mo22} that there exist examples with two switching states for which both
\[\liminf_{n \to \infty} \frac{1}{\log n} \sup_{i_1,\ldots,i_n \in \I} \log \left\|A_{i_n} \cdots A_{i_1}\right\|=0\]
and 
\[\limsup_{n \to \infty} \frac{1}{\log n} \sup_{i_1,\ldots,i_n \in \I} \log \left\|A_{i_n} \cdots A_{i_1}\right\|=1\]
and in this case the rate of marginal instability is therefore not asymptotic to any polynomial at all, even in the weakest of possible senses. 

Theorem \ref{th:main} is proved by the consideration of several sub-cases as follows. In the case where no common eigenvector exists it is classical that the system is marginally stable, and in some form this result goes back at least as far as the work of N.E. Barabanov in \cite{Ba88}. Outside this special case the switching states $A_1,\ldots,A_N$ can be made simultaneously upper triangular by choosing a basis whose first element is a common eigenvector, and if every switching state satisfies $|\det A_j|<1$ then marginal stability follows essentially by a result of N. Guglielmi and M. Zennaro \cite[Lemma 5.1]{GuZe01}. The remaining substantial case is that in which there exist two switching states $A_j$, $A_k$ such that $|\det A_j|=|\det A_k|=1$ and such that $A_j$ and $A_k$ are individually diagonalisable, simultaneously triangularisable, and not simultaneously diagonalisable. Prior to the present work this class of systems had been successfully analysed only in certain special cases in which additional algebraic relations exist between the eigenvalues (see \cite[Theorem 10]{BeCoHa16}, \cite[Remark 5.1]{GuZe01}, \cite[Lemma 5.2]{GuZe01} and remarks prior to Proposition 3 of \cite{JuPrBl08}). The complete analysis of this sub-case represents the principal new contribution necessary for the proof of Theorem \ref{th:main}.


\section{Proof of Theorem \ref{th:main}}

\subsection{Three special cases} As was indicated in the preceding section, the proof of Theorem \ref{th:main} proceeds by reduction into three principal sub-cases. Each of these is treated by one of the three propositions which follow. These special cases are organised into a complete proof of the theorem in a subsequent subsection. 

Before stating the three propositions we note the following elementary estimate which will simplify some subsequent calculations.
\begin{lemma}\label{le:only}
The Euclidean operator norm on $2 \times 2$ upper triangular matrices has the following properties:
\begin{enumerate}[(i)]
\item
Let $A \in M_2(\C)$ be upper triangular and let $A' \in M_2(\C)$ be the matrix each of whose entries is equal to the absolute value of the corresponding entry of $A$. Then $\|A\|= \|A'\|$.
\item
Let $A, B \in M_2(\R)$ be upper triangular and non-negative. If every entry of $B$ is greater than or equal to the corresponding entry of $A$, then $\|A\|\leq \|B\|$. If the upper-right entry of $B$ is strictly greater than that of $A$, then $\|A\|<\|B\|$.
\end{enumerate}
\end{lemma}
\begin{proof}
For every $a,b,c \in \C$ we have 
\[\left\|\begin{pmatrix} a&b\\0&c\end{pmatrix} \right\|^2=\rho\left(\begin{pmatrix} a&b\\0&c\end{pmatrix}\begin{pmatrix} a^*&0\\b^*&c^*\end{pmatrix}\right) = \rho\left(\begin{pmatrix} |a|^2+|b|^2&bc^*\\b^*c&|c|^2\end{pmatrix}\right)\]
and when $a, b, c$ are real and non-negative this last matrix is also real and non-negative. It follows easily from the Perron-Frobenius theorem for non-negative matrices that the spectral radius of the last matrix is non-decreasing as a function of the non-negative real variables $a$, $b$ and $c$, and this proves the first clause of (ii). For every $a,b,c \in \C$ we may further compute 
\[ \left\|\begin{pmatrix} a&b\\0&c\end{pmatrix} \right\|^2 =\frac{1}{2}\left(|a|^2+|b|^2+|c|^2\right) + \frac{1}{2}\sqrt{\left(|a|^2-|c|^2\right)^2+|b|^2\left(2|a|^2+|b|^2+2|c|^2\right)}.\]
It is immediately clear from this expression that the norm of the matrix is unchanged if $a$, $b$ and $c$ are replaced with their absolute values, which gives (i). It is also obvious that a strict increase in the value of $|b|$ results in a strict increase in the value of the norm, and this yields the second clause of (ii).\end{proof}
The following result is by now rather standard in the ``qualitative'' form
\[\sup_{n \geq 1} \max_{i_1,\ldots,i_n \in \{1,\ldots,N\}}\left\|A_{i_n}\cdots A_{i_1}\right\|<\infty\]
and in this form it can be obtained straightforwardly from such works as \cite{Ba88,El95,Ju09,Wi02} and many others besides. Since we can easily give an explicit bound on this supremum, we include such a bound for interest.
\begin{proposition}\label{pr:irr}
Let $A_1,\ldots,A_N$ be $2\times 2$ complex matrices such that the joint spectral radius $\varrho(\{A_1,\ldots,A_N\})$ is equal to $1$, and suppose that $A_1,\ldots,A_N$ do not have a common eigenspace. Then the real number
\[\kappa:=\min_{\|u\|, \|v\|= 1} \max_{X \in\{ I,A_1,\ldots,A_N\}} |\langle Xu,v\rangle|\]
is nonzero, and we have 
\[\sup_{n \geq 1} \max_{i_1,\ldots,i_n \in \{1,\ldots,N\}}\left\|A_{i_n}\cdots A_{i_1}\right\| \leq \kappa^{-1} \max_j \|A_j\|.\]
\end{proposition}
\begin{proof}
By continuity and the compactness of the unit sphere in $\C^2$ there exist unit vectors $u',v' \in \C^2$ such that 
\[ \max_{X \in \{I,A_1,\ldots,A_N\}} |\langle Xu',v'\rangle|=\min_{\|u\|, \|v\|= 1} \max_{X \in \{I,A_1,\ldots,A_N\}} |\langle Xu,v\rangle|=:\kappa.\]
If $\kappa=0$ then this implies that all of the vectors $u', A_1u',\ldots,A_Nu'$ are perpendicular to $v'$. Since the orthogonal complement of $v'$ is one-dimensional this implies that all of the vectors $u', A_1u',\ldots,A_Nu'$ are proportional to one another, which is to say that $u'$ is a shared eigenvector for $A_1,\ldots,A_N$. This contradicts the hypotheses and we conclude that $\kappa$ must be nonzero as required.

We now claim that for every pair of $2\times 2$ complex matrices $B_1$ and $B_2$,
\begin{equation}\label{eq:bee}\max_{X \in \{I, A_1, \ldots, A_N\}} \|B_1XB_2\| \geq \kappa \|B_1\|\cdot \|B_2\|.\end{equation}
By continuity it suffices to establish the result for invertible matrices $B_1$ and $B_2$. Given invertible matrices $B_1$ and $B_2$ choose unit vectors $u, v \in \C^2$ such that $\|B_1^*v\|=\|B_1\|$ and $\|B_2u\|=\|B_2\|$. Define two unit vectors by $\hat{u}:=\|B_2\|^{-1}B_2u$ and $\hat{v}:=\|B_1\|^{-1}B_1^*v$. Using the Cauchy-Schwarz inequality we obtain
\begin{align*}\max_{X \in \{I,A_1,\ldots,A_N\}} \|B_1XB_2\|  &\geq \max_{X \in \{I,A_1,\ldots,A_N\}} |\langle B_1XB_2u,v \rangle|\\
&=\max_{X \in \{I,A_1,\ldots,A_N\}} |\langle XB_2u, B_1^*v\rangle|\\
&=\|B_1\|\cdot \|B_2\|\cdot \max_{X \in \{I,A_1,\ldots,A_N\}} |\langle X\hat{u},\hat{v}\rangle|\geq\kappa \|B_1\|\cdot\|B_2\|\end{align*}
and we have proved \eqref{eq:bee}. 

Now define
\[\alpha_n:=\max_{j_1,\ldots,j_n \in \{1,\ldots,N\}} \left\|A_{j_n}\cdots A_{j_1}\right\|\]
for every $n \geq 1$. By the definition of the joint spectral radius we have $\lim_{n \to \infty} \alpha_n^{1/n}=\varrho(\{A_1,\ldots,A_N\})$. If $n,m \geq 1$ are arbitrary, choose $j_1,\ldots,j_n, k_1,\ldots,k_m \in \{1,\ldots,N\}$ such that
\[\alpha_n=\|A_{j_n}\cdots A_{j_1}\|,\qquad \alpha_m=\|A_{k_m}\cdots A_{k_1}\|.\] 
Applying \eqref{eq:bee} with $B_1:=A_{j_n}\cdots A_{j_1}$ and $B_2:=A_{k_m}\cdots A_{k_1}$ we find that
\begin{align*}\max \{\alpha_{n+m}, \alpha_{n+m+1}\} &\geq \max_{X \in \{I, A_1,\ldots,A_N\}} \|A_{j_n}\cdots A_{j_1} X A_{k_m}\cdots A_{k_1}\|\\
&\geq  \kappa \cdot \|A_{j_n}\cdots A_{j_1}\|\cdot  \|A_{k_m}\cdots A_{k_1}\|\\
&=\kappa \cdot\alpha_n \alpha_m,\end{align*}
and since obviously $\alpha_{n+m+1} \leq (\max_j \|A_j\|) \cdot \alpha_{n+m}$ we have
\[\alpha_{n+m} \geq (\max_j \|A_j\|)^{-1} \kappa\cdot \alpha_n \alpha_m\]
for every $n,m \geq 1$. It follows that if we define a further sequence $(\beta_n)$ by $\beta_n:= (\max_j \|A_j\|)^{-1}\kappa \alpha_n$ then $\beta_{n+m} \geq \beta_n \beta_m$ for all $n,m\geq 1$. Fekete's subadditivity lemma therefore applies to the sequence $(-\log \beta_n)$ and provides
\[ \sup_{n \geq 1} \beta_n^{\frac{1}{n}} = \lim_{n \to \infty} \beta_n^{\frac{1}{n}}=\lim_{n \to \infty} \alpha_n^{\frac{1}{n}}=\varrho(\{A_1,\ldots,A_N\})=1.\]
We conclude that $\beta_n \leq 1$ for every $n \geq 1$ and this is precisely the desired result.
\end{proof}
The following result is closely related to \cite[Lemma 5.2]{GuZe01}, but we use an alternative method of proof due to J. Varney (see \cite[\S3.1.2]{Va22} and \cite{MoVa22}). This method results in an explicit bound which is not present in the earlier work \cite{GuZe01}. 
\begin{proposition}\label{pr:bdd}
Let $A_1,\ldots,A_N$ be $2\times 2$ upper-triangular complex matrices and let $\lambda \in (0,1)$ and $M>0$. Suppose that:
\begin{enumerate}[(i)]
\item
The diagonal entries of every matrix $A_j$ have absolute value at most $1$;
 \item
Every matrix $A_j$ either is diagonal, or has at least one diagonal entry with absolute value smaller than $\lambda$;
\item
Every off-diagonal entry of every $A_j$ has absolute value at most $M$.
\end{enumerate}
Then 
\[\sup_{n \geq 1} \max_{i_1,\ldots,i_n \in \{1,\ldots,N\}} \|A_{i_n}\cdots A_{i_1}\| \leq 1+\frac{2M}{1-\lambda}.\]
\end{proposition}
\begin{proof}
Define three matrices by
\[B_1=\begin{pmatrix} 1&M\\0&\lambda\end{pmatrix},\qquad B_2=\begin{pmatrix} \lambda &M\\0&1\end{pmatrix},\qquad B_3=\begin{pmatrix} 1&0\\0&1\end{pmatrix}.\]
By hypothesis, for every $j=1,\ldots,N$ there exists $k=k(j) \in \{1,2,3\}$ such that every entry of the matrix $A_j$ has absolute value less than or equal to the corresponding entry of the matrix $B_{k(j)}$. For every $j=1,\ldots,N$ let us write
\[A_j=\begin{pmatrix} a_j & b_j \\ 0&c_j\end{pmatrix}.\]
If $n \geq 1$ and $j_1,\ldots,j_n \in \{1,\ldots,N\}$ are arbitrary, then by repeated applications of Lemma \ref{le:only}
\begin{align*}\left\|A_{j_n}\cdots A_{j_1}\right\| &=\left\|\begin{pmatrix} a_{j_n}&b_{j_n}\\ 0&c_{j_n} \end{pmatrix}\cdots \begin{pmatrix} a_{j_1}&b_{j_1}\\ 0&c_{j_1} \end{pmatrix}\right\| \\
& =\left\|\begin{pmatrix} a_{j_n}\cdots a_{j_1}&\sum_{\ell=1}^n a_{j_n}\cdots a_{j_{\ell+1}} b_{j_\ell} c_{j_{\ell-1}}\cdots c_{j_1} \\ 0&c_{j_n}\cdots c_{j_1} \end{pmatrix}\right\| \\
& =\left\|\begin{pmatrix} |a_{j_n}\cdots a_{j_1}|&\left|\sum_{\ell=1}^n a_{j_n}\cdots a_{j_{\ell+1}} b_{j_\ell} c_{j_{\ell-1}}\cdots c_{j_1}\right| \\ 0&|c_{j_n}\cdots c_{j_1}| \end{pmatrix}\right\| \\
& \leq\left\|\begin{pmatrix} |a_{j_n}\cdots a_{j_1}|&\sum_{\ell=1}^n |a_{j_n}\cdots a_{j_{\ell+1}} b_{j_\ell} c_{j_{\ell-1}}\cdots c_{j_1}| \\ 0&|c_{j_n}\cdots c_{j_1}| \end{pmatrix}\right\| \\
& =\left\|\begin{pmatrix} |a_{j_n}|&|b_{j_n}|\\ 0&|c_{j_n}| \end{pmatrix}\cdots \begin{pmatrix} |a_{j_1}|&|b_{j_1}|\\ 0&|c_{j_1}| \end{pmatrix}\right\| \\
&\leq \left\|B_{k(j_n)} \cdots B_{k(j_1)}\right\|.\end{align*}
We conclude that
\[\sup_{n \geq 1} \max_{i_1,\ldots,i_n \in \{1,\ldots,N\}} \|A_{i_n}\cdots A_{i_1}\| \leq \sup_{n \geq 1} \max_{j_1,\ldots,j_n \in \{1,2,3\}} \|B_{j_n}\cdots B_{j_1}\|.\]
Moreover, since $B_3$ is simply the identity matrix, its removal from a product $B_{j_n}\cdots B_{j_1}$ does not change the norm of that product. Therefore, 
\[ \sup_{n \geq 1} \max_{j_1,\ldots,j_n \in \{1,2,3\}} \|B_{j_n}\cdots B_{j_1}\|= \sup_{n \geq 1} \max_{j_1,\ldots,j_n \in \{1,2\}} \|B_{j_n}\cdots B_{j_1}\|.\]
Consider now a product $B_{j_n}\cdots B_{j_1}$ of the matrices $B_1$ and $B_2$ which has exactly $m$ instances of $B_1$ and $n-m$ instances of $B_2$. We claim that
\[\|B_{j_n}\cdots B_{j_1}\|\leq \|B_1^{m}B_2^{n-m}\|.\]
Suppose for a contradiction that $B_{j_n}\cdots B_{j_1}$ is a product which includes exactly $m$ instances of $B_1$ and $n-m$ instances of $B_2$, that $\|B_{j_n}\cdots B_{j_1}\|$ is maximal among all products with the former property, but that $\|B_{j_n}\cdots B_{j_1}\|> \|B_1^{m}B_2^{n-m}\|$. Since $B_{j_n}\cdots B_{j_1} \neq B_1^{m}B_2^{n-m}$ the former product necessarily contains an instance of the product $B_2B_1$, so that we may write
\[B_{j_n}\cdots B_{j_1}=(B_{j_n}\cdots B_{j_{\ell+3}} )B_2B_1(B_{j_\ell} \cdots B_{j_1}),\] 
say. Define $X:=B_{j_n}\cdots B_{j_1}$. The difference
\[Y:=(B_{j_n}\cdots B_{j_{\ell+3}})B_1B_2(B_{j_\ell} \cdots B_{j_1})-(B_{j_n}\cdots B_{j_{\ell+3}})B_2B_1(B_{j_\ell} \cdots B_{j_1})\]
is equal to 
\[(B_{j_n}\cdots B_{j_{\ell+3}})(B_1B_2-B_2B_1)(B_{j_\ell} \cdots B_{j_1}).\]
By direct calculation
\[B_1B_2-B_2B_1=\begin{pmatrix}\lambda&2M\\0&\lambda \end{pmatrix}-\begin{pmatrix}\lambda &2M\lambda\\0&\lambda\end{pmatrix}=\begin{pmatrix}0&2M(1-\lambda)\\0&0\end{pmatrix}\]
and it follows easily that the matrix $Y$ is non-negative and upper triangular and its upper-right entry is positive. Consequently, by Lemma \ref{le:only}(ii),
\begin{align*}\|(B_{j_n}\cdots B_{j_{\ell+3}} )B_1B_2(B_{j_\ell} \cdots B_{j_1}) \| &=\|X+Y\|\\
&>\|X\|=\|(B_{j_n}\cdots B_{j_{\ell+3}} )B_2B_1(B_{j_\ell} \cdots B_{j_1}) \|.\end{align*}
This inequality contradicts the presumed maximality of $\|B_{j_n}\cdots B_{j_1}\|$. We conclude that if 
$B_{j_n}\cdots B_{j_1}$ is composed of exactly $m$ instances of $B_1$ and $n-m$ instances of $B_2$ then necessarily
\[\|B_{j_n}\cdots B_{j_1}\|\leq \left\|B_1^{m}B_2^{n-m}\right\|\]
as claimed. Since $m$ and $n$ were arbitrary it follows that 
\[\sup_{n \geq 1} \max_{j_1,\ldots,j_n \in \{1,2\}} \|B_{j_n}\cdots B_{j_1}\|\leq \sup_{n \geq 1} \max_{0 \leq m \leq n} \left\|B_1^{m} B_2^{n-m}\right\|.\]
By direct calculation
\begin{align*}B_1^mB_2^{n-m}&=\begin{pmatrix}1 &M\sum_{\ell=0}^{m-1} \lambda^\ell\\ 0&\lambda^m  \end{pmatrix}\begin{pmatrix}\lambda^{n-m} &M\sum_{\ell=0}^{n-m-1} \lambda^\ell\\ 0&1  \end{pmatrix}\\
&=\begin{pmatrix} \lambda^m & M\left(\sum_{\ell=0}^{m-1}\lambda^\ell + \sum_{\ell=0}^{m-n-1}\lambda^\ell\right) \\ 0&\lambda^{n-m}\end{pmatrix}\end{align*}
and using Lemma \ref{le:only} once more
\begin{align*}\sup_{n \geq 1} \max_{i_1,\ldots,i_n \in \{1,\ldots,N\}} \left\|A_{i_n}\cdots A_{i_1}\right\|  &\leq  \sup_{n \geq 1} \max_{j_1,\ldots,j_n \in \{1,2\}} \left\|B_{j_n}\cdots B_{j_1}\right\|\\
&\leq \sup_{n \geq 1} \max_{0 \leq m \leq n} \left\|B_1^{m} B_2^{n-m}\right\|\\
&\leq \left\| \begin{pmatrix}1&2M\sum_{\ell=0}^\infty \lambda^\ell \\0&1\end{pmatrix}\right\|\\
&\leq \left\| \begin{pmatrix}1&0 \\0&1\end{pmatrix}\right\| + \left\| \begin{pmatrix}0&2M\sum_{\ell=0}^\infty \lambda^\ell \\0&0\end{pmatrix}\right\|\\
&=1+\frac{2M}{1-\lambda}\end{align*}completing the proof of the proposition.
\end{proof}
Our third proposition provides the core novelty needed to prove Theorem \ref{th:main}.
\begin{proposition}\label{pr:tr}
Let $A_1,A_2$ be $2 \times 2$ upper-triangular complex matrices which are diagonalisable but not simultaneously diagonalisable, which both have spectral radius equal to $1$, and which satisfy $|\det A_1|=|\det A_2|=1$. Then there exist a sequence $\sigma \colon \N \to \{1,2\}$ and vector $v \in \C^2$ such that
\[\liminf_{n \to \infty}\frac{1}{n}\left\|A_{\sigma(n)}\cdots A_{\sigma(1)}v\right\| >0.\]
\end{proposition}
\begin{proof}
Clearly both diagonal entries of both matrices are complex numbers of unit modulus. Since the conclusion is unaffected by replacing $A_1$ with $e^{-i\theta_1}A_1$ and $A_2$ with $e^{-i\theta_2}A_2$ for some $\theta_1,\theta_2 \in \R$, we assume without loss of generality that the lower-right entry of each of $A_1$ and $A_2$ is equal to $1$, so that
\[A_1=\begin{pmatrix} e^{i\phi} & a\\ 0&1\end{pmatrix},\qquad A_2=\begin{pmatrix} e^{i\psi} & b\\ 0&1\end{pmatrix},\]
say. If we had $e^{i\phi}=1$ then since $A_1$ is diagonalisable it would have to be the identity, and in any basis such that $A_2$ is diagonal, both matrices would be diagonal, contradicting the hypotheses. We therefore must have $e^{i\phi}\neq 1$ and by the same reasoning also $e^{i\psi} \neq 1$. To prove the proposition we will treat three cases in turn: that in which both $e^{i\phi}$ and $e^{i \psi}$ are roots of unity; that in which neither is a root of unity; and that in which exactly one of the two numbers is a root of unity. 

{\bf{First case.}} Suppose first that both $e^{i\phi}$ and $e^{i\psi}$ are roots of unity. The conclusion of the proposition is unaffected by a  change of basis for $\C^2$, so we will begin by rewriting the matrices $A_1$ and $A_2$ in a convenient form. By an upper-triangular change of basis for $\C^2$ we may obtain
\[A_1=\begin{pmatrix} e^{i\phi} &a'\\ 0&1\end{pmatrix},\qquad A_2=\begin{pmatrix} e^{i\psi} & 0\\ 0&1\end{pmatrix},\]
say. Since $A_1$ and $A_2$ are not simultaneously diagonalisable $a'$ cannot be zero. Making use of the fact that $e^{i\phi}-1$ and $a'$ are nonzero we have
\[\begin{pmatrix} \sqrt{\frac{e^{i\phi}-1}{a'}} &0 \\0 &1/\sqrt{\frac{e^{i\phi}-1}{a'}}   \end{pmatrix}  \begin{pmatrix} e^{i\phi} &a' \\0 &1 \end{pmatrix}\begin{pmatrix} \sqrt{\frac{e^{i\phi}-1}{a'}} &0 \\0 &1/\sqrt{\frac{e^{i\phi}-1}{a'}}   \end{pmatrix} ^{-1}=\begin{pmatrix} e^{i\phi} &e^{i\phi}-1 \\0 &1 \end{pmatrix},\]
\[\begin{pmatrix} \sqrt{\frac{e^{i\phi}-1}{a'}} &0 \\0 &1/\sqrt{\frac{e^{i\phi}-1}{a'}}   \end{pmatrix}   \begin{pmatrix} e^{i\psi} &0 \\0 &1 \end{pmatrix}\begin{pmatrix} \sqrt{\frac{e^{i\phi}-1}{a'}} &0 \\0 &1/\sqrt{\frac{e^{i\phi}-1}{a'}}   \end{pmatrix} ^{-1}=\begin{pmatrix} e^{i\psi} &0\\0 &1 \end{pmatrix}\]
for any consistently-chosen square root of $(e^{i\phi}-1)/a'$, so by a second change of basis we may assume without loss of generality that
\[A_1=\begin{pmatrix} e^{i\phi} &e^{i\phi}-1 \\0 &1 \end{pmatrix},\qquad A_2=\begin{pmatrix} e^{i\psi} &0 \\0 &1 \end{pmatrix}.\]
We observe that for every $n,m \in \Z$ 
\[A_1^n = \begin{pmatrix} e^{in\phi} &e^{in\phi}-1 \\0 &1 \end{pmatrix},\qquad A_2^m = \begin{pmatrix} e^{im\psi} &0 \\0 &1 \end{pmatrix}.\]
Choose integers $q_1,q_2 \geq 2$ such that $e^{q_1 i \phi}=e^{q_2 i \psi}=1$ and observe that $A_1^{q_1}=A_2^{q_2}=I$, so in particular $A_1^{q_1-1}=A_1^{-1}$ and $A_2^{q_2-1}=A_2^{-1}$. We directly calculate
\begin{align*}A_1A_2A_1^{-1}A_2^{-1}&= \begin{pmatrix} e^{i\phi} &e^{i\phi}-1 \\0 &1 \end{pmatrix}\begin{pmatrix} e^{i\psi} &0 \\0 &1 \end{pmatrix}\begin{pmatrix} e^{-i\phi} &e^{-i\phi}-1 \\0 &1 \end{pmatrix}\begin{pmatrix} e^{-i\psi} &0 \\0 &1 \end{pmatrix} \\
&= \begin{pmatrix} e^{i(\phi+\psi)} &e^{i\phi}-1 \\0 &1 \end{pmatrix}\begin{pmatrix} e^{-i(\phi+\psi)} &e^{-i\phi}-1 \\0 &1 \end{pmatrix} \\
&=\begin{pmatrix} 1 & e^{i\psi}-e^{i(\phi+\psi)}+e^{i\phi}-1\\ 0&1\end{pmatrix}\\
&=\begin{pmatrix}1 & (1-e^{i\phi})(e^{i\psi}-1)\\0&1\end{pmatrix}\end{align*}
and thus $A_1A_2A_1^{q_1-1}A_2^{q_2-1}=A_1A_2A_1^{-1}A_2^{-1}$ is a nontrivial Jordan matrix since the upper-right entry is necessarily nonzero. If we define $\sigma \colon \N \to \{1,2\}$ to be the periodic switching signal with period $q_1+q_2$ such that $\sigma(j)=2$ for $1 \leq j \leq q_2-1$, $\sigma(j)=1$ for $q_2 \leq j \leq q_1+q_2-2$, $\sigma(q_1+q_2-1)=2$ and $\sigma(q_1+q_2)=1$ then it is clear that
\[\liminf_{n \to \infty}\frac{1}{n}\left\|A_{\sigma(n)}\cdots A_{\sigma(1)}v\right\| >0\]
where $v \in \C^2$ is any vector with nonzero second co-ordinate. This concludes the proof of the proposition in the case where $e^{i\phi}$ and $e^{i\psi}$ are both roots of unity.

{\bf{Second case.}} We now consider the case in which neither $e^{i\phi}$ nor $e^{i\psi}$ is a root of unity. In this case we make an identical basis transformation so that 
\[A_1=\begin{pmatrix} e^{i\phi} &e^{i\phi}-1 \\0 &1 \end{pmatrix},\qquad A_2=\begin{pmatrix} e^{i\psi} &0 \\0 &1 \end{pmatrix}\]
and so that we again have for every $n,m \in \Z$
\[A_1^n = \begin{pmatrix} e^{in\phi} &e^{in\phi}-1 \\0 &1 \end{pmatrix},\qquad  A_2^m = \begin{pmatrix} e^{im\psi} &0 \\0 &1 \end{pmatrix}.\]
To obtain the conclusion we will prove the following claim: there exists $M\geq 1$ such that for every $\alpha \in \C$, there exist integers $n,m$ in the range $1 \leq n,m\leq M$ such that the vector
\[\begin{pmatrix} \beta \\ 1\end{pmatrix} :=A_1^n A_2^m \begin{pmatrix} \alpha\\1\end{pmatrix} \]
satisfies $|\beta| \geq |\alpha|+1$. Once this has been proved it follows immediately that for any initial vector $v \in \C^2$ with second co-ordinate equal to $1$, we may apply the claim inductively to construct a switching sequence  $\sigma \colon \N \to \{1,2\}$ which satisfies $\|A_{\sigma(n_k)} \cdots A_{\sigma(1)}v\| \geq k$ along a sequence of times $(n_k)_{k=1}^\infty$ such that $n_1=0$ and $2 \leq n_{k+1}-n_k \leq 2M$ for every $k \geq 1$. Given any integer $m \geq 1$, choose the largest possible $k\geq 1$ such that $n_k \leq m$; we have $0 \leq m-n_k <n_{k+1}-n_k \leq 2M$ and $m < n_{k+1} \leq 2Mk$, so
\begin{align*}\left\|A_{\sigma(m)} \cdots A_{\sigma(1)}v\right\| &\geq \left\|\left(A_{\sigma(m)}\cdots A_{\sigma(n_k+1)}\right)^{-1}\right\|^{-1}\cdot  \left\|A_{\sigma(n_k)}\cdots A_{\sigma(1)}v\right\|\\
&\geq \left(\min\left\{\left\|A_1^{-1}\right\|^{-1}, \left\|A_2^{-1}\right\|^{-1}\right\}\right)^{n_k-m} \cdot k \\
&\geq \frac{\left(\min\left\{\left\|A_1^{-1}\right\|^{-1}, \left\|A_2^{-1}\right\|^{-1}\right\}\right)^{-(2M-1)}}{2M} \cdot m.\end{align*}
Thus
\[\liminf_{m \to \infty}\frac{1}{m} \|A_{\sigma(m)}\cdots A_{\sigma(1)}v \| >0\]
and the conclusion of the proposition follows from the validity of the claim.

Let us therefore prove the claim. By Kronecker's theorem, $\{e^{in\phi} \colon n\in \N\}$ and $\{e^{in\psi} \colon n \in \N\}$ are both dense in the unit circle in $\C$. It follows that there exists $M \geq 1$ such that both of the sets $\{e^{in\phi} \colon 1 \leq n \leq M \}$ and $\{e^{in\psi} \colon 1 \leq n \leq M\}$ intersect every arc in the unit circle whose angular length is at least $2\pi/3$. Given $\alpha=r e^{i\theta} \in \C$, choose $m$ such that $1 \leq m \leq M$ and such that
\[e^{im\psi} \in \left\{e^{i \vartheta} \colon \vartheta \in \left(-\theta - \frac{\pi}{3}, -\theta+\frac{\pi}{3}\right)\right\}\]
and then choose $n$ such that $1 \leq n \leq M$ and such that
\[e^{in\phi} \in \left\{e^{i \vartheta} \colon \vartheta \in \left(-\theta-m\psi + \frac{2\pi}{3}, -\theta-m\psi+\frac{4\pi}{3}\right)\right\}.\]
We then have
\[e^{i(m\psi+\theta)} \in \left\{e^{i \vartheta} \colon \vartheta \in \left( - \frac{\pi}{3}, \frac{\pi}{3}\right)\right\} \]
and
\[e^{i(m\psi+n\phi+\theta)} \in \left\{e^{i \vartheta} \colon \vartheta \in \left(  \frac{2\pi}{3}, \frac{4\pi}{3}\right)\right\} \]
so that
\[\Re\left(e^{-i(m\psi+\theta)}\right)= \Re\left(e^{i(m\psi+\theta)}\right)=\cos(m\psi+\theta)  >\frac{1}{2},\]
and
\[\Re\left(e^{-i(m\psi+n\phi+\theta)}\right)=\Re\left(e^{i(m\psi+n\phi+\theta)}\right)=\cos(m\psi+n\phi+\theta)<-\frac{1}{2}.\]
We calculate that
\begin{align*}\begin{pmatrix}\beta\\1\end{pmatrix}&=A_1^n A_2^m \begin{pmatrix}re^{i\theta}\\1\end{pmatrix}\\
 &= \begin{pmatrix} e^{in\phi} &e^{in\phi}-1 \\0 &1 \end{pmatrix} \begin{pmatrix} e^{im\psi} &0 \\0 &1 \end{pmatrix}\begin{pmatrix}re^{i\theta}\\1\end{pmatrix}=\begin{pmatrix}e^{in \phi}(1+re^{i(m\psi+\theta)})-1
\\1
\end{pmatrix}\end{align*}
and since
\begin{align*}|\beta|=\left|e^{in \phi}\left(1+re^{i(m\psi+\theta)}\right)-1\right|&=\left|r+e^{-i(m\psi+\theta)}-e^{-i(m\psi+n\phi+\theta)}\right|\\
& \geq \Re\left(r+e^{-i(m\psi+\theta)}-e^{-i(m\psi+n\phi+\theta)}\right)\\
&=r+\Re\left(e^{-i(m\psi+\theta)}\right) -\Re\left(e^{-i(m\psi+n\phi+\theta)}\right)\\
&>r+1\\
&=|\alpha|+1\end{align*}
this proves the claim. The conclusion of the proposition in the second case now follows by inductive application of the claim.

{\bf{Third case.}} It remains only to treat the case in which one of $e^{i\phi}$ and $e^{i\psi}$ is a root of unity and the other is not. Using the symmetry between the two matrices
\[A_1=\begin{pmatrix} e^{i\phi} & a\\ 0&1\end{pmatrix},\qquad A_2=\begin{pmatrix} e^{i\psi} & b\\ 0&1\end{pmatrix},\]
by swapping the labels of the two matrices if necessary we may without loss of generality assume that $e^{i\psi}$ is a root of unity and that $e^{i\phi}$ is not. But then the two matrices
\[B_1:=A_1=\begin{pmatrix} e^{i\phi} & a\\ 0&1\end{pmatrix},\qquad B_2:=A_1A_2=\begin{pmatrix} e^{i(\psi+\phi)} & e^{i\phi}b+a\\ 0&1\end{pmatrix}\]
are both upper-triangular with lower-right entry equal to $1$ and with upper-left entry lying on the unit circle but not a root of unity. We observe that $B_1$ and $B_2$ cannot be simultaneously diagonalised, since if they were diagonal in some basis then both $A_1=B_1$ and $A_2=A_1^{-1}A_1A_2=B_1^{-1}B_2$ would be diagonal in that basis which contradicts the hypotheses of the proposition. We may therefore apply the already-proved second case of the proposition to $B_1$ and $B_2$ to obtain a switching sequence $\hat\sigma \colon \N \to \{1,2\}$ and vector $v \in \C^2$ such that
\[\liminf_{n \to \infty} \frac{1}{n}\left\|B_{\hat\sigma(n)}\cdots B_{\hat\sigma(1)}v\right\| >0.\]
By replacing every instance of $2$ in the switching sequence $\hat\sigma$ with an instance of $2$ followed by an additional instance of $1$, it is clear that we may define a new switching sequence $\sigma \colon \N \to \{1,2\}$ and sequence of times $(n_k)_{k=1}^\infty$ such that $n_1 \in \{1,2\}$ and such that we have both $A_{\sigma(n_k)}\cdots A_{\sigma(1)}v= B_{\sigma(k)}\cdots B_{\sigma(1)}v$ and $1 \leq n_{k+1}-n_k \leq 2$ for every $k \geq 1$. In particular $k \leq n_k \leq 2k$ for every $k \geq 1$ and it follows easily that 
\[\liminf_{n \to \infty} \frac{1}{n} \left\|A_{\sigma(n)}\cdots A_{\sigma(1)}v\right\| \geq \frac{\left\|A_1^{-1}\right\|^{-1}}{2} \cdot \liminf_{n \to \infty} \frac{1}{n}\left\|B_{\hat\sigma(n)}\cdots B_{\hat\sigma(1)}v\right\| >0.\]
We have proved the proposition in the third case, and the proof of the proposition is now complete.
\end{proof}

\subsection{The proof of Theorem \ref{th:main}}

We may now prove Theorem \ref{th:main}. Let $A_1,\ldots,A_N$ be $2\times 2$ complex matrices and suppose that their joint spectral radius is equal to $1$. If $A_1,\ldots,A_N$ do not have a common eigenspace then 
\[\sup_{n \geq 1} \max_{i_1,\ldots,i_n \in \{1,\ldots,N\}} \left\|A_{i_n}\cdots A_{i_1}\right\|<\infty\]
by Proposition \ref{pr:irr} as required, so for the remainder of the proof we assume the existence of a common eigenvector $v_1 \in \C^2$. 

Suppose that  $\{A_j \colon |\det A_j|=1\}$ is simultaneously diagonalisable. If this set is nonempty let $v_2 \in \C^2$ be a common eigenvector of these matrices which is not proportional to $v_1$; if the set is empty, choose arbitrarily a vector $v_2 \in \C^2$ which is not proportional to $v_1$. In the basis $(v_1,v_2)$ for $\C^2$ the matrices $A_1,\ldots,A_N$ are all upper triangular. Their diagonal entries are their eigenvalues, and it follows by Gelfand's formula that the absolute value of their eigenvalues is bounded above by the joint spectral radius of $A_1,\ldots,A_N$, which is $1$. Thus  the diagonal entries of the matrices $A_1,\ldots,A_N$ are all less than or equal to $1$ in absolute value. If there exists at least one matrix $A_j$ such that $|\det A_j|<1$, let $\lambda \in (0,1)$ be the maximum of $\sqrt{|\det A_j|}$ among all such matrices, and if no such matrix exists let $\lambda \in (0,1)$ be arbitrary. Let $M\geq 0$ be the maximum of the absolute values of the off-diagonal entries of the matrices $A_1,\ldots,A_N$. For every $A_j$ such that $|\det A_j|<1$, one of the diagonal entries must have absolute value less than or equal to $\lambda$ and the other clearly has absolute value less than or equal to $1$; and the entry above the diagonal has absolute value at most $M$. On the other hand every $A_j$ such that $|\det A_j|=1$ is diagonal in the basis $(v_1,v_2)$ and has both nonzero entries bounded above by $1$ in absolute value.  The hypotheses of Proposition \ref{pr:bdd} are therefore satisfied by $A_1,\ldots,A_N$ and consequently
\[\sup_{n \geq 1} \max_{i_1,\ldots,i_n \in \{1,\ldots,N\}} \left\|A_{i_n}\cdots A_{i_1}\right\|<\infty\]
as required.

Finally suppose that $\{A_j \colon |\det A_j|=1\}$ is not simultaneously diagonalisable. Choose an orthonormal basis $(u_1,u_2)$ for $\C^2$ in which $u_1$ is proportional to $v_1$. Since this basis is related to the standard basis by a unitary transformation, in this basis the norms $\|A_j\|$ all take their original values. Clearly in this basis every $A_j$ is upper triangular. 
If there exists $A_j$ such that $|\det A_j|=1$ and $A_j$ is not diagonalisable then $A_j$ is a nontrivial Jordan matrix and we trivially have
\[\lim_{n \to \infty} \frac{1}{n}\left\|A^n_jv\right\|>0\]
for every $v \in \C^2$ not proportional to $v_1$. Otherwise, every $A_j$ which satisfies $|\det A_j|=1$ is diagonalisable.  It is easily seen that if every pair of matrices $A_k$, $A_\ell$ satisfying $|\det A_k|=|\det A_\ell|$ can be simultaneously diagonalised then all matrices $A_j$ such that $|\det A_j|=1$ can be diagonalised simultaneously, which we know is not the case, so there must exist two matrices $A_k$, $A_\ell$ satisfying $|\det A_k|=|\det A_\ell|$ which are not simultaneously diagonalisable. 
The spectral radii of these two matrices are at most $1$ by earlier reasoning and are therefore exactly $1$ by consideration of the determinant. We may therefore apply Proposition \ref{pr:tr} to $A_k$ and $A_\ell$ to conclude that there exist a switching sequence $\hat\sigma \colon \N \to \{k,\ell\}$ and vector $v \in \C^2$ such that
\[\liminf_{n \to \infty} \frac{1}{n}\|A_{\hat\sigma(n)} \cdots A_{\hat\sigma(1)}v\|>0.\]
To complete the proof of the theorem we must show that
\begin{equation}\label{eq:fbd}\max_{i_1,\ldots,i_n \in \{1,\ldots,N\}} \left\|A_{i_1}\cdots A_{i_n}\right\| \leq  1+n \cdot \max_j \|A_j\|\end{equation}
for every $n \geq 1$. We argue similarly to the proof of Proposition \ref{pr:bdd}. Let us write
\[A_j=\begin{pmatrix} a_j & b_j \\ 0&c_j\end{pmatrix}\]
for every $j=1,\ldots,N$. For every $n \geq 1$ and $j_1,\ldots,j_n$ we clearly have
\begin{align*}\left\|A_{j_n}\cdots A_{j_1}\right\|&=\left\|\begin{pmatrix} a_{j_n}&b_{j_n}\\ 0&c_{j_n} \end{pmatrix}\cdots \begin{pmatrix} a_{j_1}&b_{j_1}\\ 0&c_{j_1} \end{pmatrix}\right\|\\
& =\left\|\begin{pmatrix} a_{j_n}\cdots a_{j_1}&\sum_{\ell=1}^n a_{j_n}\cdots a_{j_{\ell+1}} b_{j_\ell} c_{j_{\ell-1}}\cdots c_{j_1} \\ 0&c_{j_n}\cdots c_{j_1} \end{pmatrix}\right\|\\
& =\left\|\begin{pmatrix} |a_{j_n}\cdots a_{j_1}|&\left|\sum_{\ell=1}^n a_{j_n}\cdots a_{j_{\ell+1}} b_{j_\ell} c_{j_{\ell-1}}\cdots c_{j_1}\right| \\ 0&|c_{j_n}\cdots c_{j_1}| \end{pmatrix}\right\|\\
& \leq\left\|\begin{pmatrix} |a_{j_n}\cdots a_{j_1}|&\sum_{\ell=1}^n |a_{j_n}\cdots a_{j_{\ell+1}} b_{j_\ell} c_{j_{\ell-1}}\cdots c_{j_1}| \\ 0&|c_{j_n}\cdots c_{j_1}| \end{pmatrix}\right\|\\
& \leq\left\|\begin{pmatrix} 1 &\sum_{\ell=1}^n |b_{j_\ell}| \\ 0&1 \end{pmatrix}\right\|\\
& \leq\left\|\begin{pmatrix} 1 &0\\ 0&1 \end{pmatrix}\right\|+\left\|\begin{pmatrix} 0 &\sum_{\ell=1}^n |b_{j_\ell}| \\ 0&0 \end{pmatrix}\right\|\\
&\leq1+n \cdot \max_j |b_j|\\
&\leq1+n \cdot \max_j \|A_j\|\\
\end{align*}
where we have used Lemma \ref{le:only} and the fact that $\max_j \max\{|a_j|,|c_j|\}\leq 1$. The inequality \eqref{eq:fbd} is proved and we have completed the proof of the theorem.

\section{Acknowledgements}

This work was partially supported by Leverhulme Trust Research Project Grant RPG-2016-194.

\bibliographystyle{acm}
\bibliography{persimmon}

\end{document}